\documentclass{amsart}

\usepackage{
 amsmath, 
 amsxtra, 
 amsthm, 
 amssymb, 
 booktabs,
 comment,
 longtable,
 mathrsfs, 
 mathtools, 
 multirow,
 stmaryrd,
 tikz-cd, 
 bbm,
 xr,
 color,
 xcolor}
 \usepackage[bbgreekl]{mathbbol}
\usepackage[all]{xy}
\usepackage[nobiblatex]{xurl}
\usepackage{hyperref}

\newtheorem{theorem}{Theorem}[section]
\newtheorem{lemma}[theorem]{Lemma}

\newtheorem{proposition}[theorem]{Proposition}
\newtheorem{corollary}[theorem]{Corollary}
\newtheorem{defn}[theorem]{Definition}

\newtheorem{lthm}{Theorem} 

\theoremstyle{remark}
\newtheorem{remark}[theorem]{Remark}

\newcommand{\Gal}{\mathrm{Gal}}
\newcommand{\BSymb}{\mathrm{BSymb}}
\newcommand{\eval}{\mathrm{eval}}
\newcommand{\Hom}{\mathrm{Hom}}
\newcommand{\Symb}{\mathrm{Symb}}
\newcommand{\cG}{\mathcal{G}}
\newcommand{\SL}{\mathrm{SL}}

\newcommand{\vp}{\varphi}
\newcommand{\GL}{\mathrm{GL}}
\newcommand{\Div}{\mathrm{Div}}

\newcommand{\ord}{\mathrm{ord}}
\newcommand{\ZZ}{\mathbb{Z}}
\newcommand{\CC}{\mathbb{C}}
\newcommand{\NN}{\mathbb{N}}
\newcommand{\QQ}{\mathbb{Q}}

\newcommand{\Zp}{\ZZ_p}

\definecolor{Green}{rgb}{0.0, 0.5, 0.0}

\renewcommand{\Im}{\mathrm{Im}}

\DeclareMathSymbol\dDelta  \mathord{bbold}{"01}
\usepackage[utf8]{inputenc}
\numberwithin{equation}{section}
\title[Ramanujan's tau function and Eisenstein series]{Lambda-invariants of Mazur--Tate elements attached to Ramanujan's tau function and congruences with Eisenstein series}

\author[A.~Doyon]{Anthony Doyon}
\address[Doyon]{Anthony Doyon\newline D\'epartement de Math\'ematiques\\
C\'egep Beauce-Appalaches, Saint-Georges\\
1055 116e rue\\
Saint-Georges, QC\\
Canada G5Y 3G1}
\email{adoyon@cegepba.qc.ca}

\author[A.~Lei]{Antonio Lei}
\address[Lei]{Antonio Lei\newline Department of Mathematics and Statistics\\University of Ottawa\\
150 Louis-Pasteur Pvt\\
Ottawa, ON\\
Canada K1N 6N5}
\email{antonio.lei@uottawa.ca}



\begin{document}
\begin{abstract} Let $p\in\{3,5,7\}$ and let $\Delta$ denote the weight twelve modular form arising from Ramanujan's tau function. We show that $\Delta$ is congruent to an Eisenstein series $E_{k,\chi, \psi}$ modulo $p$ for explicit choices of $k$ and  Dirichlet characters $\chi$ and $\psi$. We then prove formulae describing the Iwasawa invariants of the Mazur--Tate elements attached to $\Delta$, confirming numerical data gathered by the authors in a previous work.
\end{abstract}

\maketitle

\section{Introduction}\label{sec:intro}
Let $\Delta \in \mathcal{S}_{12}(\text{SL}_2(\ZZ))$ be the unique normalized cuspidal modular form of weight twelve and level one given by
$$\Delta(z) = \sum_{n \geq 1} \tau(n) q^n, $$
where $q = e^{2 \pi i z}$ for some complex variable $z \in \CC$ in the upper half-plane and $\tau: \NN \to \ZZ$ is Ramanujan's tau function defined by the identity $$\sum_{n \geq 1} \tau(n) q^n = q \prod_{n \geq 1}(1-q^n)^{24}.$$

Given a prime $p$, the Mazur--Tate elements attached to $\Delta$ over $\QQ(\mu_{p^\infty})$ defined in \cite{MT} interpolate complex $L$-values twisted by Dirichlet characters of $p$-power conductor and encode important arithmetic information on $\Delta$. After an appropriate normalization, they give rise to elements of quotients of the Iwasawa algebra. One particularly important arithmetic invariant of these elements is the Iwasawa $\lambda$-invariant, which gives the number of zeros of these elements in the $p$-adic open unit disk (see Definition~\ref{def:mu-lambda}).

In the case where $\Delta$ is $p$-ordinary, these invariants are related to the $p$-adic $L$-function of $\Delta$.
When $p\in\{3,5,7\}$, $\Delta$ is non-ordinary at $p$ and it is not clear (to the authors at least) whether the $\lambda$-invariants are related to a $p$-adic $L$-function. In \cite[Theorems~1 and 2]{PW}, Pollack--Weston showed that the $\lambda$-invariants of the Mazur--Tate elements at a non-ordinary prime $p$ often satisfy a regular pattern under appropriate hypotheses.  Although results in \textit{loc. cit.} do not apply to $\Delta$ at  $p\in\{3,5,7\}$,  numerical calculations carried out in \cite{doyon-lei} suggest that the $\lambda$-invariants do satisfy a very regular pattern. The purpose of this article is to study these Iwasawa invariants by relating $\Delta$ to Eisenstein series modulo $p$. The first result that we prove is:

\begin{lthm}\label{thm:cong-intro}
    There is a full congruence of Fourier coefficients between $\Delta$ and a weight two Eisenstein series $E_p$ modulo $p$, where $E_3 = E_{2, \omega_3, \omega_3}$, $E_5 = E_{2,\omega_5^2,\mathbbm{1}}$ and $E_7 = E_{2, \omega_7^4, \mathbbm{1}}$  are defined in Theorem~\ref{thm:eis series} (here, $\omega_p$ denotes the Teichmüller character of conductor $p$).
\end{lthm}

Theorem~\ref{thm:cong-intro} is a special case of Theorem~\ref{thm:cong}, where we prove that $\Delta$ is in fact congruent to an explicit infinite family of Eisenstein series.
Although multiplicity one modulo $p$ results such as \cite[Theorem~3.11]{bel-pol} do not apply in the current setting (see Remark~\ref{rk:mult-one-mod-p}), the congruences exhibited in Theorem~\ref{thm:cong-intro} suggest that there exists a link between the modular symbol $\varphi_\Delta^+$ of sign $+1$ associated with $\Delta$ and the boundary modular symbol $\phi_p$ associated to the Eisenstein series $E_p$ modulo $p$. One advantage of working with modular symbols attached to an Eisenstein series is that they can be described explicitly, allowing us to prove:

\begin{lthm}\label{thm:lambda57-Eisen-intro}
    Let $p \in \{5,7\}$. Let $\theta_{n,p}$ denote the Mazur--Tate element of level $n$ (see Definition~\ref{def:MT}) attached to the Eisenstein series $E_p$. Then, the Iwasawa invariants of these elements (see Definition~\ref{def:mu-lambda}) are given by $$\mu(\theta_{n,p}) =0,\quad\lambda(\theta_{n,p}) = p^n - 1.$$
\end{lthm}

Theorem~\ref{thm:lambda57-Eisen-intro} is a special case of Theorem~\ref{thm:lambda inv}, where we give a simple sufficient condition for an Eisenstein series to satisfy the formula stated above. 

The computations in \cite{doyon-lei} show that the $\lambda$-invariants of the Mazur--Tate elements attached to $\Delta$ at $p\in\{5,7\}$ satisfy exactly the same formula given in Theorem~\ref{thm:lambda57-Eisen-intro} for small $n$. We prove that this formula, in fact, holds for all $n$. As we have already alluded to above,  multiplicity one modulo $p$ results are not available in our current setting. Instead, we shall take a computational approach. Another theoretical issue we face is that the Fontaine--Laffaille condition  (which requires that the weight of the modular form to be smaller than the prime $p$) does not hold in our current setting. Consequently, the canonical period of Vatsal \cite{vatsal}, which behaves well under congruences, is not available. We overcome this by considering cohomological periods instead, following the approach taken in \cite{PW}. 

We explicitly compute the modular symbols attached to $\Delta$ (denoted $\vp_\Delta^+$ in the main text) using the methods developed in \cite[\S~8]{stein} and \cite{bel-pol}, which are implemented in SageMath.  We then establish an explicit congruence modulo $p$ between $\varphi_{\Delta}^+$ and a certain boundary symbol arising from the Eisenstein series given in Theorem~\ref{thm:lambda57-Eisen-intro}; see Theorem~\ref{thm:congmodsymb} for a precise statement. Combining this congruence with Theorem~\ref{thm:lambda57-Eisen-intro} allows us to compute the Iwasawa invariants associated with $\Delta$.

\begin{lthm}[{Corollary~\ref{cor:p=5,7}}]\label{thm:lambda-delta-57}
    Fix $p \in \{5,7\}$ and let $\theta_{n,\Delta}$ denote the level $n$ Mazur--Tate element attached to $\Delta$. Then $$\mu (\theta_{n,\Delta}) =0,\quad\lambda (\theta_{n,\Delta}) = p^n-1.$$
\end{lthm}

The case $p = 3$ needs to be treated differently, since the $3$-adic $\mu$-invariant of the Mazur--Tate elements attached to $\Delta$ is strictly positive (the authors have mistakenly claimed that these $\mu$-invariants were zero in \cite[\S4]{doyon-lei}, which we rectify here). To compute $\lambda(\theta_{n,\Delta})$, we establish an explicit congruence between $\varphi_{\Delta}^+$ and a boundary symbol of level $\Gamma_1(27)$ (see \S~\ref{ssec:p=3}). Methods similar to those employed in the proof of Theorem~\ref{thm:lambda-delta-57} then allow us to prove:

\begin{lthm}[{Theorem~\ref{thm:3}}]\label{thm-intro:3}
    Let $p = 3$ and $n \geq 1$ be an integer. We have$$\mu(\theta_{n,\Delta}) = 3,\quad\lambda(\theta_{n,\Delta}) = p^n  - 2.$$
\end{lthm}

The code used to perform the numerical calculations for this article is available on \cite{link1}.

\subsection*{Acknowledgement}
The authors thank Rob Pollack for answering many of their questions during the preparation of the article. The authors are also indebted to the anonymous referees for helpful comments and suggestions on earlier versions of the article. This article forms part of the master thesis of the first named author at Universit\'e Laval, who was supported by an NSERC and a FRQNT graduate scholarships. The research of the second named author is supported by the NSERC Discovery Grants Program RGPIN-2020-04259 and RGPAS-2020-00096. 

\section{Eisenstein series and boundary modular symbols}\label{sec:eis}
The goal of this section is to review the definition of boundary modular symbols and to define such symbols corresponding to Eisenstein series. We closely follow the exposition of \cite[\S~1-2]{bel-das}.

Let GL$_2(\QQ)$ act on the left on $\mathbb{P}^1(\QQ)$ by linear fractional transformations. Denote by $\dDelta$ and $\dDelta^0$ the group of divisors of $\mathbb{P}^1(\QQ)$ and its subgroup of divisors of degree $0$, respectively. Let $\Gamma$ be the congruence subgroup $\Gamma_0(M)$ or $\Gamma_1(M)$ for some integer $M \geq 1$. There is a natural action of $\Gamma$ on $\dDelta$ and $\dDelta^0$. If $V$ is an abelian group endowed with a right $\Gamma$-action,  we define a right $\Gamma$-action on Hom$(\Delta^0, V)$ by: $$\varphi\mid_{\gamma} (D) = \varphi (\gamma \cdot D)\mid_{ \gamma} \quad \text{for } \gamma \in \Gamma, \varphi \in \text{Hom}(\dDelta^0, V).$$

\begin{defn}\label{defn:modsymbBD}
    A $V$-valued modular symbol on $\mathbb{P}^1(\QQ)$ for $\Gamma$ is a $\Gamma$-invariant homomorphism from $\dDelta^0$ to $V$, i.e. an element of $\Hom_{\Gamma}(\dDelta^0, V)$. We denote the group of $V$-valued \textbf{modular symbols} by $$\Symb_{\Gamma}(V) = \Hom_{\Gamma}(\dDelta^0, V).$$
\end{defn}

We now review basic notions of Hecke operators. Fix a submonoid $S \subseteq \text{GL}_2(\QQ)$ containing $\Gamma$ and let $W$ be a right $S$-module. For each $s \in S$, $\Gamma$ acts on the double coset $\Gamma s \Gamma$ by multiplication on the left. Therefore, we have a partition of $\Gamma s \Gamma$ into orbits $$\Gamma s \Gamma = \bigsqcup_{i=1}^n \Gamma s_i, \quad$$ for some $n \geq 1$ and some fixed representatives $s_i \in S$.

\begin{remark}
    The decomposition of $\Gamma s \Gamma$ into such orbits is finite (see, for example, \cite[\S~5.1]{DS}).
\end{remark}

\begin{defn}
    For each double coset $\Gamma s \Gamma$, there is a morphism
\begin{align*}
    [\Gamma s \Gamma]: W^\Gamma & \to W^\Gamma\\
    w & \mapsto w\mid_{[\Gamma s \Gamma]} = \sum_{i=1}^n w\mid_{ s_i},
\end{align*}
to which we will refer as a \textbf{Hecke operator}.
\end{defn}

\begin{remark}
    The morphism above is well-defined and independent of the choice of representatives $s_i$ (see, for example, \cite[Exercise~5.1.3]{DS}).
\end{remark}

In the cases we are interested in, $W$ will always be a $K$-vector space, where $K$ is a field of characteristic $0$ and the $S$-action on $K$ will be $K$-linear as in \cite{bel-das}. In particular, the Hecke operators will also be $K$-linear. 

The following Hecke operators will be of particular interest to us.

\begin{defn}
    If $\Gamma_1(N) \subseteq \Gamma \subseteq \Gamma_0(N)$ for some integer $N > 0$, we set: \begin{align*}
        T_\ell & = \left[ \Gamma \begin{pmatrix} 1 & 0 \\ 0 & \ell
        \end{pmatrix} \Gamma \right] \quad \text{if } \ell \nmid N,\\
        U_\ell & = \left[ \Gamma \begin{pmatrix} 1 & 0 \\ 0 & \ell
        \end{pmatrix} \Gamma \right] \quad \text{if } \ell \mid N,\\
        \langle a \rangle & = [\Gamma \gamma_a \Gamma]
    \end{align*}
    where $\gamma_a \in \Gamma_0(N)$ is any matrix whose upper-left entry is congruent to $a$ modulo $N$.
\end{defn}

Denote by $W_{\text{univ}}$ the $\QQ$-vector space of maps $S \to \QQ$ endowed with its natural right $S$-action. Define the Hecke algebra $\mathcal{H}(S, \Gamma)$ to be the $\QQ$-algebra generated by all Hecke operators acting on $W_{\text{univ}}$. Therefore, we can endow any right $S$-module $W$ with an action of $\mathcal{H}(S,\Gamma)$ if we let $[\Gamma s \Gamma]_{W_{\text{univ}}}$ act via $[\Gamma s \Gamma]_{W}$.

We now move on to the construction of boundary symbols, following \cite[\S~2.6]{bel-das}. There is a tautological short exact sequence of abelian groups \begin{equation}\label{eqn:ses}
    0 \to \dDelta^0 \to \dDelta \to \ZZ \to 0.
\end{equation} Since this sequence splits, we can form the following exact sequence of $S$-modules $$0 \to V \to \text{Hom}(\dDelta, V) \to  \text{Hom}(\dDelta^0, V) \to 0$$ by taking the $\text{Hom}(-,V)$ functor of (\ref{eqn:ses}). On taking $\Gamma$-cohomology, we obtain the following exact sequence of $\mathcal{H}(S,\Gamma)$-modules: \begin{equation}\label{eqn:longcohom}
    0 \xrightarrow{} V^\Gamma \xrightarrow{} \text{Hom}_{\Gamma}(\dDelta,V) \xrightarrow{b} \text{Symb}_\Gamma(V) \xrightarrow{h} {H}^1(\Gamma,V).
\end{equation}
Here, we have identified $\text{Symb}_\Gamma(\mathcal{V}_k(L))$ with $ H^1_c(\Gamma, \mathcal{V}_k(L))$ (see \cite[Proposition~4.2]{ash-ste}).

\begin{defn}
    The map $b$ in \eqref{eqn:longcohom} is called the \textbf{boundary map} and its image, denoted by $\BSymb_\Gamma (V)$, is called the module of \textbf{boundary modular symbols} (or simply \textbf{boundary symbols}).
\end{defn}

The exact sequence (\ref{eqn:longcohom}) yields an isomorphism of $\mathcal{H}(S, \Gamma)$-modules $$\text{BSymb}_\Gamma(V) \cong \text{Hom}_{\Gamma} (\dDelta, V)/ V^\Gamma,$$
relating modular symbols to boundary symbols.
Furthermore, there is a short exact sequence of $\mathcal{H}(S, \Gamma)$-modules $$0 \to \text{BSymb}_\Gamma(V) \to \text{Symb}_\Gamma(V) \to H^1(\Gamma, V).$$

We now introduce a notation that will be used throughout the article. For a fixed integer $k \geq 0$ and a number field $L$, we denote by $\mathcal{M}_{k+2}(\Gamma, L)$ the set of all modular forms of weight $k+2$ for the congruence subgroup $\Gamma$ with Fourier coefficients in $L$. Similarly, we denote the associated subsets of cusp forms and Eisenstein series of $\mathcal{M}_{k+2}(\Gamma, L)$ by $\mathcal{S}_{k+2}(\Gamma, L)$ and $\mathcal{E}_{k+2}(\Gamma, L)$ respectively.

As in \cite[\S~2.5]{bel-das}, we denote by $\mathcal{P}_k(L)$ the space of polynomials with coefficients in $L$ in the variable $z$ with degree at most $k$. We endow $\mathcal{P}_k(L)$ with a left GL$_2(\QQ)$-action given by 
$$(\gamma \cdot P)(z) = (a-cz)^k P\left(\frac{dz-b}{a-cz}\right), \quad \text{where } \quad \gamma = \begin{pmatrix} a & b\\ c & d \end{pmatrix} \in \text{GL}_2(\QQ).$$ 
This way, $\mathcal{V}_k(L) = \text{Hom}_L(\mathcal{P}_k(L), L)$ is equipped with a natural right GL$_2(\QQ)$-action
$$f\mid_\gamma (P) = f(\gamma \cdot P) \quad \text{for } \gamma \in \text{GL}_2(\QQ),\text{ } P \in \mathcal{P}_k(L).$$

We recall the following result, which is essentially due to Eichler and Shimura.

\begin{theorem}\label{thm:eic-shi}
Assume that the congruence subgroup $\Gamma$ satisfies $\Gamma_1(M) \subseteq \Gamma \subseteq \Gamma_0(M)$ for some integer $M$. Then, after possibly replacing $L$ by a suitable finite extension, the following holds:
\begin{enumerate}
    \item There is a short exact sequence of $\mathcal{H}(\GL_2^+(\QQ), \Gamma)$-modules 
    $$0 \to \BSymb_\Gamma(\mathcal{V}_k(L)) \to \Symb_\Gamma(\mathcal{V}_k(L)) \stackrel{h}\to H^1(\Gamma, \mathcal{V}_k(L)) \to \mathcal{E}_{k+2}(\Gamma, L) \to 0. $$ \smallskip
    \item There is an isomorphism of $\mathcal{H}(\GL_2^+(\QQ), \Gamma)$-modules $$\Im(h) \cong \mathcal{S}_{k+2}(\Gamma,L)^2.$$\smallskip
    \item There is an isomorphism of $\mathcal{H}(\GL_2^+(\QQ), \Gamma)$-modules $$\BSymb_\Gamma(\mathcal{V}_k(L)) \cong \mathcal{E}_{k+2}(\Gamma, L)$$ compatible with the Hecke operators $T_\ell$ for $\ell \nmid M$, $U_\ell$ for $\ell \mid M$ and $\langle a \rangle$ for $a \in (\ZZ / M \ZZ)^\times$.
\end{enumerate}
\end{theorem}

\begin{proof}
    See \cite[Proposition~2.5]{bel-das} for a brief discussion of the argument and \cite[\S~8]{Shi} for the original proof.
\end{proof}

Roughly speaking, Theorem~\ref{thm:eic-shi} tells us that
$$\text{Symb}_\Gamma (\mathcal{V}_k(L)) \cong \mathcal{M}_{k+2}(\Gamma, L) \oplus \mathcal{S}_{k+2}(\Gamma, L) $$
and that the contribution of the Eisenstein part $\mathcal{E}_{k+2}(\Gamma,L)\subset \mathcal{M}_{k+2}(\Gamma, L) $ to the direct sum above can be accounted for by the boundary modular symbols BSymb$_\Gamma(\mathcal{V}_k(L))$. 

We now turn our attention to the space of Eisenstein series. The following theorem can be found in \cite{miyake}; see \cite[Proposition~2.8]{bel-das}.
\begin{theorem}\label{thm:eis series}
Let $k\ge2$ and $M\ge1$ be integers. Let $\psi$ and $\chi$ be primitive Dirichlet characters of respective conductors $Q$ and $R$,  such that $QR = M$, $\psi \chi (-1) = (-1)^k$, and if $\psi = \chi = \mathbbm{1}$, assume further that $k \neq 2$. 
Let $E_{k, \psi, \chi}$ be the Fourier expansion defined by $$E_{k, \psi, \chi}(q) = c_{k,\psi,\chi,0} + \sum_{n \geq 1} c_{k,\psi,\chi,n} q^n,$$ where $$c_{k,\psi,\chi,n} = \sum_{d \mid n, d>0} \psi (n/d) \chi (d) d^{k-1}$$ for $n \geq 1$ and $$c_{k,\psi,\chi,0} = \left\{ \begin{array}{cc}
     0 \quad  & \text{if } Q >1, \\
     \frac{- B_{k, \chi}}{2k} \quad & \text{if } Q = 1.
\end{array}\right. $$ (Here, $B_{k, \chi}$ is the $k$-th generalized Bernoulli number attached to $\chi$.)
Then $E_{k, \psi, \chi}$ is a new Eisenstein series of weight $k$ and level $M$.
\end{theorem}

These Eisenstein series can be used to form an explicit basis for $\mathcal{E}_{k+2}(\Gamma_1(N),L)$, where $M|N$ (see \cite[Proposition~2.8]{bel-das}). We can describe the corresponding elements in BSymb$_{\Gamma_1(M)}(\mathcal{V}_k(L))$ explicitly.

\begin{defn}
    Let $M$ and $k \geq 2$ be positive integers. For any pair of coprime integers $(u,v)$, define $\phi_{k,u,v} \in \Hom_{\Gamma_1(M)}(\dDelta, \mathcal{V}_k(L))$ to be supported on the $\Gamma_1(M)$-orbit of the divisor $\{u/v\} \in \dDelta$ by $$\phi_{k,u,v}\left( \gamma \left( \frac{u}{v} \right)\right)(P(z)) = P \left( \gamma \left( \frac{u}{v}\right)\right) \cdot (cu+dv)^k, \quad \gamma = \begin{pmatrix}
    a & b\\
    c & d
    \end{pmatrix}
    \in \Gamma_1(M).$$
\end{defn}
We recall the following result form \cite[Proposition~2.9]{bel-das}:

\begin{proposition}\label{prop:bsymb}
Let $\psi$, $\chi$ and $k$ be as in Theorem~\ref{thm:eis series}. The boundary modular symbol $\phi_{k, \psi, \chi} \in \Hom_{\Gamma_1(M)}(\dDelta, \mathcal{V}_k)$ defined by
   $$\phi_{k, \psi, \chi} = \sum_{\substack{x \text{ }(\text{mod } Q) \\ (x,Q)=1}} \sum_{\substack{y \text{ }(\text{mod } R) \\ (y,R)=1}} \psi^{-1}(x) \chi(y) \phi_{k, x, Qy}$$
     corresponds to the Eisenstein series $E_{k+2,\psi, \chi}$ under the isomorphism given by Theorem~\ref{thm:eic-shi}(3).
\end{proposition}

\begin{remark}
    Each summand $\psi^{-1}(x) \chi(y) \phi_{k, x, Qy}$ appearing in Proposition~\ref{prop:bsymb} depends only on the $\Gamma_1(M)$-cusp of $\{x/Qy\}$, i.e. it depends only on $x$ (mod $Q$) and $y$ (mod $R$).
\end{remark}

\section{Relating \texorpdfstring{$\Delta$}{} to Eisenstein series and boundary symbols}

In \cite[\S~6, Table~5]{doyon-lei}, we obtained numerical evidence suggesting that the $\lambda$-invariants of the level $m$ Mazur--Tate elements attached to the modular form $\Delta \in \mathcal{S}_{12}(\SL_2(\ZZ))$ introduced in \S~\ref{sec:intro}  were $p^m-1$ for $p \in \{ 5,7 \}$ and $p^m-2$ for $p=3$. In this section, we present a proof of these formulae by relating the modular symbols of $\Delta$ to boundary symbols modulo $p$ (resp. $p^2$) when $p\in\{5,7\}$ (resp. $p=3$). Our congruences are established by calculations based on the methods presented in \cite[\S~2]{pol-ste} and \cite[\S~8]{stein}.

\subsection{Congruences between \texorpdfstring{$\Delta$}{} and Eisenstein series}\label{sec:modsymbdelta} 

The goal of this section is to show that the modular form $\Delta$  is congruent to a family of ordinary new Eisenstein series introduced in Theorem~\ref{thm:eis series}, of which Theorem~\ref{thm:cong-intro} in the introduction is a special case.

\begin{defn}
Let $p$ be a prime number. We write $\omega_p$ for the Teichmüller character of conductor $p$. Given integers $0\le a,b\le p-2$ and $k\ge2$ such that $(a,b,k)\ne(0,0,2)$, we write 
    \[
E_{k,a,b} = \sum_{n \geq 0} c_{k,a,b,n} q^n
    \]
    for the Eisenstein series $E_{k,\omega_p^a,\omega_p^b}$.
\end{defn}
Note that the notation $E_{k,a,b}$ implicitly depends on the choice of $p$. We shall treat the indices $a$ and $b$ as elements in $\ZZ/(p-1)\ZZ$. For example, we have $E_{k,1,0}=E_{k,1,p-1}$.

\begin{theorem}\label{thm:cong}
For $p \in \{3,5 \}$ and all integers $n \geq 0$, we have
$$\tau(n) \equiv c_{k, a, b, n} \pmod p \quad \text{if} \quad \begin{cases} 
a = 1  \text{ and } b + k \equiv 3 \pmod {p -1},\\
a = 2  \text{ and } b + k \equiv 2 \pmod { p - 1}\text{ with }(k,p)\ne(2,3). 
\end{cases}
$$
Similarly for $p=7$, we have 
$$
\tau(n) \equiv c_{k, a, b, n} \pmod 7 \quad \text{if} \quad 
\begin{cases}
a = 1  \text{ and }  b + k \equiv 5 \pmod  6,\\
a = 4  \text{ and }  b + k \equiv 2 \pmod  6.
\end{cases}
$$
\end{theorem}

\begin{proof}
Let $\ell$ be a prime number. Recall that 
\[
\tau(\ell)  \equiv\begin{cases} \ell^2 + \ell^9 & \pmod{3^3},\\
  \ell + \ell^{10}& \pmod{5^2},\\
  \ell + \ell^4 &\pmod 7,
 \end{cases}
\]
for all prime numbers $\ell$  (see for example \cite[\S~2]{serre-tau}). By Fermat's Little Theorem, it implies that
\[
    \tau(\ell) \equiv \begin{cases}
    \ell + \ell^2 & \pmod{3},\\
    \ell + \ell^2 & \pmod{5},\\
   \ell + \ell^4 &\pmod  7.
   \end{cases}
\]

By definition, for $p \in \{3,5,7\}$ and any integer $n\ge1$, we have 
\begin{align*}
c_{k,a,b,\ell} & = \omega_p^a(\ell)\omega_p^b(1)+\omega_p^a(1)\omega_p^b(\ell)\ell^{k-1}\\
& \equiv \ell^{a} +\ell^{b+k-1} \pmod p
\end{align*}
since $\omega_p(\ell) \equiv \ell$ $(\text{mod } p)$.

Therefore, our prescribed choices of $a,b$ and $k$ give $$c_{k,a,b,\ell} \equiv \tau(\ell) \pmod p$$ by Fermat's little theorem.
As $\Delta$ and $E_{k,a,b}$ are normalized Hecke eigenforms, their coefficients satisfy the same recurrence relations (see, for example, \cite[Proposition~5.8.5]{DS}). This tells us that
\[
\tau(n)\equiv c_{k,a,b,n},\quad n\ge1.
\]
Since the constant terms of both modular forms are $0$, the congruence holds for all $n\ge0$.
\end{proof}

\begin{remark}
    Theorem~\ref{thm:cong} tells us that for each $p \in \{3,5,7\}$, there is a full congruence of Fourier coefficients modulo $p$ between $\Delta$ and infinitely many Eisenstein series of the form $E_{k, a, b}$. Our proofs of Theorems~\ref{thm:lambda-delta-57} and \ref{thm-intro:3} are based on comparing the Mazur--Tate elements of $\Delta$ with the boundary symbols attached to certain Eisenstein series. Although Theorem~\ref{thm:cong} is not directly used in these proofs, it serves as a guide in our choice of Eisenstein series.
\end{remark}

\begin{remark}
    For each $p\in\{5,7\}$, we have found two choices of  $(a,b)$ for which $\tau(n)\equiv c_{k,a,b,n}\pmod p$. Each choice corresponds to a Hida family of Eisenstein series as $k$ varies. When $p=3$, there is one such family. The authors thank the referees for pointing this out to us.
\end{remark} 

\subsection{Evaluation map on modular symbols}\label{S:evaluation}
The goal of this section is to define an evaluation map on modular symbols, which will be used to establish Theorems~\ref{thm:lambda-delta-57} and \ref{thm-intro:3}  in subsequent sections. 

We begin by introducing a second point of view on modular symbols (which can be found in \cite{PW}), which is dual to that of \cite{bel-das} previously described in \S \ref{sec:eis}. Let $R$ be any commutative ring and, for any integer $k \geq 0$, let $V_k(R)$ be the space of homogeneous polynomials of degree $k$ in the variables $X$ and $Y$ with coefficients in $R$. We endow $V_k(R)$ with a right action of $\GL_2(R)$ via $$(P\mid_\gamma)(X,Y) = P(dX - cY, -bX + aY),$$ where $P \in V_k(R)$ and $\gamma = \begin{pmatrix} a & b \\ c & d \end{pmatrix} \in \GL_2(R)$.
Further, if we let $\SL_2(\ZZ)$ act on $\dDelta^0$ by linear fractional transformation, we can endow $\Hom(\dDelta^0, V_{k}(R))$ with a right action of $\SL_2(\ZZ)$ via
$$(\varphi \mid_{\gamma})(D) = (\varphi(\gamma \cdot D))\mid_{\gamma},$$
where $\varphi \in \Hom(\dDelta^0, V_{k}(R))$, $\gamma \in \SL_2(\ZZ)$ and $D \in \dDelta^0$.

\begin{defn}\label{defn:modsymb}
    Let $\Gamma \leq \SL_2(\ZZ)$ be a congruence subgroup. A \textbf{classical modular symbol} of weight $k$ and level $\Gamma$ is an element of $\Hom_{\Gamma}(\dDelta^0, V_k(R))$ for some commutative ring $R$.
\end{defn}

\begin{remark}
    This definition differs from Definition~\ref{defn:modsymbBD} as the spaces $V_k(R)$ and $\mathcal{V}_k(R)$ are duals of each other and the actions on these spaces are slightly different. This will not be an issue for our purposes because we will mainly focus on modular symbols of weight $0$ where these two definitions coincide. We chose to introduce both definitions as boundary modular symbols are generally defined as elements of the spaces in Definition~\ref{defn:modsymb} (see \cite{bel-das} for example) whilst modular symbols attached to modular forms are generally defined as elements of the spaces in Definition~\ref{defn:modsymbBD} (see \cite{PW}).
\end{remark}

Let $\Gamma \leq \SL_2(\ZZ)$ be a fixed congruence subgroup. To any modular form ${f \in \mathcal{S}_k(\Gamma)}$, one can associate a modular symbol of sign $+1$ denoted by $\varphi_f^+ \in \Hom_{\Gamma}(\dDelta^0,V_{k-2}(\CC))$ defined by 
$$\varphi_f^+(\{r\} - \{s\}) = 2 \pi i \int_s^r f(z)(zX + Y)^{k-2}dz,$$
where the integration takes place on the semi-circular path joining $r$ and $s$ contained in the upper half plane (or a vertical line if either $r$ or $s$ is $\infty$). In this article, we require a canonical way of normalizing the modular symbols so that their values on every divisor $D \in \dDelta^0$ are $p$-integral for $p \in \{3,5,7\}$ and that at least one of these values is coprime to $p$. For this purpose, we assume from now on that all modular symbols are normalized by the cohomological period of \cite[Definition~2.1]{PW}. We refer the reader to the code \cite[\texttt{mu.sagews}]{link1}.


Given a commutative ring $R$, we have the evaluation map
\begin{align*}
    \eval: V_k(R) & \to R\\
    P(X,Y) & \mapsto P(0,1).
\end{align*} It induces a projection $$\alpha: \Hom(\dDelta^0, V_k(R)) \to \Hom(\dDelta^0, R).$$
In other words, given a homomorphism (not necessarily a modular symbol) in $\Hom(\dDelta^0,V_{k}(R))$, we obtain a homomorphism in $\Hom(\dDelta^0,R)$ via the map $\eval$.

 When $R=\ZZ / N \ZZ$ for some integer $N>1$, we write $\eval_N$ for the evaluation map $\eval$ and $\alpha_N$ for the induced map $\alpha$ from $ \Hom(\dDelta^0, V_k(\ZZ/N\ZZ))$ to $\Hom(\dDelta^0, \ZZ/N\ZZ)$.
We show below that under certain hypotheses, $\alpha_N$ sends modular symbols of weight $k$ to weight $0$ modular symbols.
\begin{proposition}\label{prop:wtktowt2}
    If $\Gamma$ is a congruence subgroup containing $\Gamma_1(N)$, the restriction of $\alpha_N$ to $\Hom_{\Gamma}(\dDelta^0, V_k(\ZZ / N \ZZ))$ takes values in $\Hom_{\Gamma_1(
    N)}(\dDelta^0, \ZZ / N\ZZ)$. In other words, $\alpha_N$ induces a map $$ \Hom_{\Gamma}(\dDelta^0, V_k(\ZZ / N \ZZ)) \to \Hom_{\Gamma_1(N)}(\dDelta^0, \ZZ / N \ZZ).$$
\end{proposition}

\begin{proof}
    Let $\varphi \in \Hom_{\Gamma}(\dDelta^0, V_k(\ZZ / N \ZZ))$ and $\gamma \in \Gamma_1(N)$. For every $D \in \dDelta^0$, \begin{equation}\label{eqn:pfprojn}
        \varphi(D) = (\varphi \mid_{\gamma})(D) =\varphi(\gamma \cdot D) \mid_{\gamma}
    \end{equation} in $V_k(\ZZ / N \ZZ)$. One sees that for every $\gamma = \begin{pmatrix} a & b \\ c & d \end{pmatrix} \in \Gamma_1(N)$ and every polynomial $P(X,Y) \in V_k(\ZZ / N \ZZ)$, the following relation holds: \begin{align*}
        \eval_N(P(X,Y)\mid_\gamma) & \equiv \eval_N(P(dX - cY, -bX + aY)) \pmod{N}\\ & \equiv \eval_N(P(X, -bX + Y)) \pmod{N}\\ 
        & \equiv P(0,1) \equiv \eval_N(P(X,Y)) \pmod{N}.
    \end{align*} 
    Thus, applying $\eval_N$ on both sides of \eqref{eqn:pfprojn} gives
    $$(\alpha_N \varphi)(D) = (\alpha_N \varphi)(\gamma \cdot D) = (\alpha_N \varphi)\mid_{\gamma}(D),$$ which completes the proof.
\end{proof}

\subsection{Modular symbols at \texorpdfstring{$p \in\{5,7\}$}{}}\label{ssec:57}
We  specialize to the case $p \in\{ 5,7\}$. To prove Theorem~\ref{thm:lambda-delta-57}, we show that the modular symbol attached to $\Delta$ and the boundary symbol arising from one of the two weight 2 Eisenstein series in Theorem~\ref{thm:cong} are congruent modulo $p$. When $p=5$, we work with $E_{2,2,0}$, whereas $E_{2,4,0}$ is utilized for the case $p=7$. Note that these choices give Eisenstein series of level $p$. In particular, we will study the values of the aforementioned boundary symbols at divisors of the form $\{\frac{a}{p^{n+1}}\}$, where $p\nmid a$.

One may ask whether one could work with $E_{2,1,1}$ (resp. $E_{2,1,3}$) when $p=5$ (resp. $p=7$) instead. These Eisenstein series are of level $p^2$. For our purposes, the formula given by Proposition~\ref{prop:bsymb} would require us to evaluate $\phi_{0,x,py}$ at divisors of the form $\{\frac{a}{p^{n+1}}\}$, which would be $0$ when $M=p^2$. Therefore, this would not yield any information on the Iwasawa invariants of the Mazur--Tate elements of $\Delta$.

Proposition~\ref{prop:wtktowt2} tells us that ${\alpha_p {\varphi}_{\Delta}^+ \in \Hom_{\Gamma_1(p)}(\dDelta^0,\ZZ / p \ZZ)}$ is a $\ZZ / p\ZZ = V_0(\ZZ / p \ZZ)$-valued modular symbol of weight $0$ and level $\Gamma_1(p)$. 
The results of Pollack--Stevens \cite[Corollary~2.7, Corollary~2.10]{pol-ste} and Stein \cite[Theorem~8.4]{stein} tell us that every modular symbol in $\Hom_{\Gamma_1(p)}(\dDelta^0, \ZZ / p \ZZ)$ is completely determined by its values on a finite subset $S_p \subseteq \dDelta^0$. In particular, two symbols in $\Hom_{\Gamma_1(p)}(\dDelta^0,\ZZ / p \ZZ)$ that agree on every element of $S_p$ are equal. This allows us to prove:

\begin{theorem}\label{thm:congmodsymb}
    Let $p\in \{5,7\}$. If we denote the reduction modulo $p$ of the symbol $\varphi$ by $\overline{\varphi}$, then there exists non-zero constants $c_p \in \ZZ / p \ZZ$ such that $$\alpha_5 \overline{\varphi}_\Delta^+(\{r\} - \{s\}) \equiv c_5 ( \overline{\phi_{0,\omega_5^2, \mathbbm{1}}}(\{r\})(P) - \overline{\phi_{0,\omega_5^2, \mathbbm{1}}}(\{s\})(P)) \pmod{5},$$ $$\alpha_7 \overline{\varphi}_\Delta^+(\{r\} - \{s\}) \equiv c_7 ( \overline{\phi_{0,\omega_7^4, \mathbbm{1}}}(\{r\})(P') - \overline{\phi_{0,\omega_7^4, \mathbbm{1}}}(\{s\})(P')) \pmod{7}$$ for all divisors $\{r\}, \{s\} \in \dDelta$ and where $P \in \ZZ / 5 \ZZ = \mathcal{P}_0(\ZZ / 5 \ZZ )$ and $P' \in \ZZ / 7 \ZZ = \mathcal{P}_0(\ZZ / 7 \ZZ)$.
\end{theorem}

\begin{proof}
    The calculations in the following are based on the methods discussed in \cite{pol-ste} and \cite[\S8]{stein}, which are carried out on SageMath. The code of our calculations can be found in \cite[\texttt{lambda.sagews}]{link1}. 
    
    First, one computes that a modular symbol in $\Hom_{\Gamma_1(5)}(\dDelta^0, \ZZ / 5 \ZZ)$ is completely determined by its values on $$S_5 = \left\{ \left\{ \frac{-2}{5} \right\} - \left\{ \frac{-1}{3}\right\}, \left\{ \frac{1}{5} \right\} - \left\{ \frac{1}{4}\right\}, \left\{ \frac{-1}{3}\right\} - \left\{ \frac{-1}{4} \right\} \right\}.$$ Similarly, a modular symbol in $\Hom_{\Gamma_1(7)}(\dDelta^0, \ZZ / 7 \ZZ)$ is completely determined by its values on \begin{align*}
        S_7 = & \left\{ \left\{ \frac{-2}{7}\right\} - \left\{ \frac{-1}{4} \right\}, \left\{  \frac{-3}{7}\right\} - \left\{ \frac{-2}{5} \right\}, \left\{ \frac{1}{7} \right\} - \left\{ \frac{1}{6} \right\}, \right.\\
        & \left. \left\{ \frac{4}{25}\right\} - \left\{ \frac{1}{6}\right\}, \left\{ \frac{-1}{5} \right\} - \left\{ \frac{-1}{6} \right\} \right\}.
    \end{align*} The code that generates these sets has already been implemented in SageMath following the methods of \cite[\S~8]{stein}. One can verify numerically that the congruences in the statement of the theorem hold for all divisors in $S_p$ and find explicitly the constants $c_5 = 2$ and $c_7 = 1$.
    Note that the reduction modulo $p$ of the relevant boundary symbols can be calculated directly using Proposition~\ref{prop:bsymb}: \begin{equation*}
        \overline{\phi_{0,\omega_5^2,\mathbbm{1}}}(\{r\})(P) = P \cdot \begin{cases} 2, & \text{if } \{r\} \in \Gamma_1(5)\{\infty\},\\
        3, & \text{if } \{r\} \in \Gamma_1(5)\left\{\frac{2}{5}\right\},\\
        0, & \text{otherwise,}
        \end{cases}
    \end{equation*}
    
    \begin{equation*}
        \overline{\phi_{0,\omega_7^4, \mathbbm{1}}}(\{r\})(P') = P' \cdot \begin{cases}
        2, & \text{if } \{r\} \in \Gamma_1(7)\{\infty\},\\
        1, & \text{if } \{r\} \in \Gamma_1(7)\left\{ \frac{2}{7} \right\},\\
        4, & \text{if } \{r\} \in \Gamma_1(7)\left\{ \frac{3}{7} \right\},\\
        0, & \text{otherwise.}
        \end{cases}
    \end{equation*}
    Our calculations on SageMath tell us that 
 \begin{align*}
\alpha_5 {\varphi}_\Delta^+\left(\left\{ \frac{-2}{5}\right\} - \left\{ \frac{-1}{3}\right\}\right) &\equiv 1 \pmod{5}, \\
        \alpha_5{\varphi}_\Delta^+\left(\left\{ \frac{1}{5}\right\} - \left\{ \frac{1}{4}\right\}\right) &\equiv 4 \pmod{5}, \\
    \alpha_5{\varphi_\Delta^+}\left(\left\{ \frac{-1}{3}\right\} - \left\{ \frac{-1}{4}\right\}\right) &\equiv 0 \pmod{5}, \\
    \alpha_7 {\varphi_\Delta^+}\left(\left\{ \frac{-2}{7}\right\} - \left\{ \frac{-1}{4}\right\}\right) &\equiv 1 \pmod{7}, \\
    \alpha_7 {\varphi_\Delta^+}\left(\left\{ \frac{-3}{7}\right\} - \left\{ \frac{-2}{5}\right\}\right) &\equiv 4 \pmod{7}, \\
    \alpha_7 {\varphi_\Delta^+}\left(\left\{ \frac{1}{7}\right\} - \left\{ \frac{1}{6}\right\}\right) &\equiv 2 \pmod{7}, \\
    \alpha_7 {\varphi_\Delta^+}\left(\left\{ \frac{4}{25}\right\} - \left\{ \frac{1}{6}\right\}\right) &\equiv \alpha_7 {\varphi_\Delta^+}\left(\left\{ \frac{-1}{5}\right\} - \left\{ \frac{-1}{6}\right\}\right) \equiv 0 \pmod{7},
\end{align*}
  yielding the desired result.
\end{proof}

\begin{remark}\label{rk:mult-one-mod-p}
    With an appropriate multiplicity one modulo $p$  result such as \cite[Theorem~3.11]{bel-pol}, one could potentially obtain a more direct proof for Theorems~\ref{thm:congmodsymb}. It would be sufficient to show that the Hecke eigenvalues of the  boundary symbols and  $\alpha_p{\varphi}_\Delta^+$ are congruent up to a fixed non-zero constant. 
    
    Note that \textit{loc. cit.}  allows us to compare modular forms of the same weight whose levels are of the form $Np$, where $p\nmid N$. The weight 12 Eisenstein series obtained from Theorem~\ref{thm:cong} have level $p^2$. Therefore, we may not apply  \cite[Theorem~3.11]{bel-pol} directly to compare the modular symbols of $\Delta$ and the corresponding weight 12 Eisenstein series. 
    
    One may be tempted to work with $\alpha_p{\varphi}_\Delta^+$ and the weight 2 Eisenstein series in Theorem~\ref{thm:cong}. While these Eisenstein series do have level $p$ when $p\in\{5,7\}$, the hypothesis (Good Eisen) in \cite[\S 3.1]{bel-pol} only applies to maximal ideals corresponding to weight 2 Eisenstein series of the form $E_{2,\mathbb{1},\psi}$, where $\psi$ is a Dirichlet character of conductor prime to $p$ (see the proof of Lemma 3.1 in \textit{op. cit.}). The weight 2 Eisenstein series in Theorem~\ref{thm:cong} are visibly not of this form, so \cite[Theorem~3.11]{bel-pol} does not apply.
    \end{remark}

\subsection{Modular symbols at \texorpdfstring{$p = 3$}{}}\label{ssec:p=3}

The case $p=3$ requires different considerations from the cases $p \in\{ 5,7\}$. Under the normalization discussed in \S\ref{S:evaluation}, we find that $\alpha(\varphi_\Delta^+)$ sends the elements of $\dDelta^0$ to either $3^2\ZZ_3$ or $3^3\ZZ_3$. Furthermore, the relevant values used in the computation of $\lambda$-invariants, namely $\alpha{\varphi}_\Delta^+(\{\infty\} - \{a / 3^n \})$ (see Definition~\ref{def:MT} below) belong to $3^3\ZZ_3$. This tells us that the $3$-adic $\mu$-invariants are at least $3$.

It follows from Proposition~\ref{prop:wtktowt2} that  $\alpha_{81}\varphi_\Delta^+$ is an element of $\Hom_{\Gamma_1(81)}(\dDelta^0,\ZZ/81\ZZ)$. Our discussion in the previous paragraph says that we have, in fact 
\begin{equation}
    \alpha_{81}{\varphi_\Delta^+}\in\Hom_{\Gamma_1(81)}(\dDelta^0,3^2\ZZ/81\ZZ)\cong \Hom_{\Gamma_1(81)}(\dDelta^0,\ZZ/9\ZZ).\label{eq:iso-81}
\end{equation}
We write $\beta_9 {\varphi}_{\Delta}^+\in \Hom_{\Gamma_1(81)}(\dDelta^0,\ZZ/9\ZZ) $ for the image of $\alpha_{81}\varphi_\Delta^+$ under \eqref{eq:iso-81}.

We have found by brute force that $\beta_9 {\varphi}_{\Delta}^+$ in fact belongs to the subgroup  $$\Hom_{\Gamma_1(27)}(\dDelta^0, \ZZ / 9 \ZZ)\subseteq \Hom_{\Gamma_1(81)}(\dDelta^0,\ZZ/9\ZZ).$$ More explicitly, we define $\phi_9 \in \Hom_{\Gamma_1(27)}(\dDelta, \ZZ / 9 \ZZ)$ by assigning to $D\in \dDelta$ the values\begin{equation*}
 \begin{cases} 0, & \text{if } D \in \Gamma_1(27)\left\{ \frac{1}{9} \right\} \bigcup \Gamma_1(27)\left\{ \frac{8}{27} \right\} \bigcup \Gamma_1(27)\left\{ \frac{10}{27} \right\} \bigcup \Gamma_1(27)\left\{ \frac{8}{9} \right\} \bigcup \Gamma_1(27)\left\{ \infty \right\}, \\
    1, & \text{if } D \in \Gamma_1(27)\left\{ 0 \right\} \bigcup \Gamma_1(27)\left\{ \frac{1}{12} \right\} \bigcup \Gamma_1(27)\left\{ \frac{1}{10} \right\} \bigcup \Gamma_1(27)\left\{ \frac{1}{8} \right\} \bigcup \Gamma_1(27)\left\{ \frac{5}{12} \right\}, \\
    3, & \text{if } D \in \Gamma_1(27)\left\{ \frac{2}{27} \right\} \bigcup \Gamma_1(27)\left\{ \frac{2}{9} \right\} \bigcup \Gamma_1(27)\left\{ \frac{7}{27} \right\} \bigcup \Gamma_1(27)\left\{ \frac{11}{27} \right\} \bigcup \Gamma_1(27)\left\{ \frac{7}{9} \right\},\\
    4, & \text{if } D \in \Gamma_1(27)\left\{ \frac{1}{11} \right\} \bigcup \Gamma_1(27)\left\{ \frac{1}{7} \right\} \bigcup \Gamma_1(27)\left\{ \frac{1}{3} \right\} \bigcup \Gamma_1(27)\left\{ \frac{1}{2} \right\} \bigcup \Gamma_1(27)\left\{ \frac{2}{3} \right\},\\
    6, & \text{if } D \in \Gamma_1(27)\left\{ \frac{4}{27} \right\} \bigcup \Gamma_1(27)\left\{ \frac{5}{27} \right\} \bigcup \Gamma_1(27)\left\{ \frac{4}{9} \right\} \bigcup \Gamma_1(27)\left\{ \frac{13}{27} \right\} \bigcup \Gamma_1(27)\left\{ \frac{5}{9} \right\},\\
    7, & \text{if } D \in \Gamma_1(27)\left\{ \frac{1}{13} \right\} \bigcup \Gamma_1(27)\left\{ \frac{1}{6} \right\} \bigcup \Gamma_1(27)\left\{ \frac{1}{5} \right\} \bigcup \Gamma_1(27)\left\{ \frac{1}{4} \right\} \bigcup \Gamma_1(27)\left\{ \frac{5}{6} \right\}.
    \end{cases}
\end{equation*}
Note that this is indeed a boundary symbol since it is invariant under $\Gamma_1(27)$ by construction.
It should correspond to the reduction modulo 9 of a linear combination of certain Eisenstein series of level $27$. However, an explicit description is not required for our purposes. 

We prove the following analogue of Theorem~\ref{thm:congmodsymb}. 

\begin{theorem}\label{thm:congmodsymb3}
  The congruence 
  $$\beta_9\varphi_\Delta^+(\{r\} - \{s\}) \equiv \phi_9(\{r\})(P) - \phi_9(\{s\})(P) \pmod{9}$$ holds for all divisors $\{r\}, \{s\} \in \Div(\mathbb{P}^1(\QQ))$ and $P \in \ZZ / 9 \ZZ = \mathcal{P}_0(\ZZ / 9 \ZZ)$.
\end{theorem}

\begin{proof}
    This follows from the same argument as in the proof of Theorem~\ref{thm:congmodsymb}. The numerical computations carried out can be found in \cite[\texttt{lambda.sagews}]{link1}. With SageMath and the methods of \cite{pol-ste} and \cite{stein}, one can explicitly compute a set of divisors $S_{27} \subseteq \Delta^0$ with $\# S_{27} = 55$ such that any element of $\Hom_{\Gamma_1(27)}(\dDelta^0, \ZZ / 9 \ZZ)$ is completely determined by its image on $S_{27}$. Since both $\beta_9{\varphi}_{\Delta}^+$ and the image of $\phi_9$ by the boundary map belong to $\Hom_{\Gamma_1(27)}(\dDelta^0, \ZZ / 9 \ZZ)$, the claim follows from verifying that these symbols agree on all elements of $S_{27}$ (see Table~\ref{tab:long}).
\end{proof}

\section{Formulae of Iwasawa \texorpdfstring{$\lambda$}{}-invariants for \texorpdfstring{$\Delta$}{}}
We  first recall the definition of Mazur--Tate elements given in \cite[\S~2.1]{PW}. Then, using the congruence of modular symbols we have obtained in \S~\ref{ssec:57} and \S~\ref{ssec:p=3}, we explicitly compute the $p$-adic $\lambda$-invariants of the Mazur--Tate elements attached to $\Delta$ at $p \in \{ 3,5,7 \}$. 

As in \cite[\S~3.1]{doyon-lei}, for $p$ an odd prime, let $G_n = \Gal(\QQ(\mu_{p^n}) / \QQ)$. We identify an element $a \in (\ZZ / p^n \ZZ)^\times$ with the unique element $\sigma_a \in G_n$ such that $\sigma_a(\zeta) = \zeta^a$ for all $\zeta \in \mu_{p^n}$. We denote by $K_n$ the unique extension of $\QQ$ contained in $\QQ(\mu_{p^{n+1}})$ of degree $[K_n:\QQ] = p^n$. We write $\mathcal{G}_n = \Gal(K_n / \QQ)$, which we can identify with a quotient of $G_{n+1}$ via the natural projection map $\pi_n: G_{n+1} \twoheadrightarrow \mathcal{G}_n$. 

Following \cite[\S~2.1]{PW}, we define Mazur--Tate elements attached to a  modular symbol  as follows.

\begin{defn}\label{def:MT}
Fix integers $n, k\ge 0$, a congruence subgroup $\Gamma \leq \SL_2(\ZZ)$ and a commutative ring $R$. Let $\varphi \in \Hom_{\Gamma}(\dDelta^0, V_k(R))$ be a modular symbol. We define  $$\Theta_{n,\varphi} = \sum_{a \in (\ZZ / p^{n+1} \ZZ)^\times}  \varphi (\{\infty\} - \{a/p^{n+1}\}) \big |_{(X,Y) = (0,1)} \cdot \sigma_a \in \CC[G_{n+1}]$$
and denote the image of $\Theta_{n,\varphi}$ in $\CC[\cG_n]$ under the natural norm map induced by $\pi_n$ by $\tilde\Theta_{n,\varphi}$.

The Mazur--Tate element of level $n$ attached to $\varphi$ is defined to be
\[
\theta_{n,\varphi}=\frac{\tilde\Theta_{n,\varphi}}{\Omega_\varphi},
\]
where $\Omega_\varphi$ is the cohomological period for $\varphi$ given in \cite[Definition~2.1]{PW}. 

If $f \in \mathcal{S}_k(\Gamma)$ is a fixed modular form, we denote $\theta_{n, \varphi_{f}^+}$ by $\theta_{n,f}$.
\end{defn}

\begin{remark}
 In \cite[Definition~2.1]{PW},  two periods $\Omega_\varphi^+$ and $\Omega_\varphi^-$ are utilized, depending on whether the modular symbol lies inside the $+1$ or $-1$ eigenspace of the involution. In the present article, we have projected the modular symbols being considered from the group ring of $G_{n+1}$ to that of $\cG_n$. Consequently, we obtain modular symbols in the $+1$ eigenspace. For simplicity, we have omitted the superscript $+$ from our notation.
\end{remark}

\begin{remark}
    Let $R$ be any commutative ring, and $\phi \in \BSymb_\Gamma(R) =$ \linebreak $\Hom_{\Gamma}(\dDelta, \mathcal{V}_0(R))$ be a weight $0$ boundary symbol for some congruence subgroup $\Gamma$. One can extend the definition of level $n$ Mazur--Tate elements to such a symbol in the following way. After fixing an element $P \in R$, the function \begin{align*}
        \Phi_P: \dDelta & \to R\\
        \{r\} & \mapsto \phi(\{r\})(P)
    \end{align*}
    is a classical modular symbol in the sense of \cite[\S~2.1]{PW} after restricting it to $\dDelta^0$, i.e. $\Phi_P \in \Hom_\Gamma(\dDelta^0, V_0(R))$. Therefore, we can adapt the definition of Mazur--Tate elements to $\phi$ by setting $$\Theta_{n,\phi} = \sum_{a \in (\ZZ / p^{n+1} \ZZ)^\times} (\Phi_1(\{\infty\}) - \Phi_1(\{a / p^{n+1}\})) \cdot \sigma_a \in \mathbb{C}[G_{n+1}].$$
\end{remark}

\begin{remark}
    For the case of $\Delta$, the normalization by the cohomological period forces that $\theta_{n,\Delta} \in \ZZ_p[\mathcal{G}_n]$.
\end{remark}

We now briefly review the definitions of the Iwasawa $\mu$ and $\lambda$ invariants of an element $F\in\Zp[\cG_n]$. We choose a generator $\gamma_n$ of the Galois group $\cG_n$. We may write $F$ as a polynomial $\sum_{i=0}^{p^n-1}a_iT^i$, where $T=\gamma_n-1$.

\begin{defn}\label{def:mu-lambda}
For a non-zero element $F =\sum_{i=0}^{p^n-1} a_i T^i\in \Zp[\cG_n]$, we define the $\mu$ and $\lambda$-invariants of $F$ by
\begin{align*}
    \mu(F)& = \min\limits_{i} \ord_p (a_i),\\
    \lambda (F)& = \min \{i:\ord_p (a_i) = \mu (F) \},
\end{align*}
where $\ord_p$ denotes the $p$-adic valuation on $\ZZ$. When $F=0$, we set $$\mu(F)=\lambda(F)=\infty.$$
\end{defn}
\begin{remark}
 The definitions above are independent of the choice of the generator $\gamma_n$. 
\end{remark}

Under the normalization given as in Definition~\ref{def:MT}, our computations suggest that $\mu(\theta_{n,\Delta})=1$ (resp. $0$) when $p=3$ (resp. $p\in\{5,7\}$) for all $n\ge0$. We shall prove this as part of Theorem~\ref{thm:3} (resp. Corollary~\ref{cor:p=5,7}). Our proof of Corollary~\ref{cor:p=5,7} relies on a formula for the  $\lambda$-invariants of the Mazur--Tate elements attached to weight $0$ boundary symbols of level $\Gamma_1(p)$ (Theorem~\ref{thm:lambda inv}). To prove this, we require the following preliminary lemma.

\begin{lemma}\label{lem:cusps}
    Let $p$ be an odd prime number and $n \geq 1$ be an integer. If $\phi \in \BSymb_{\Gamma_1(p)}(\mathcal{V}_0(\ZZ / p \ZZ))$, then 
    \[
\phi(\{a/p^n\})(1)=\phi(\{a/p\})(1)
    \]
    for all $a\in(\ZZ/p^n\ZZ)^\times$.
\end{lemma}
\begin{proof}
     As $\phi \in \text{BSymb}_{\Gamma_1(p)}(\mathcal{V}_0(\ZZ / p \ZZ))$, the values of $\phi(\{r\})(1) \pmod{p}$ for $\{r\} \in \dDelta$ are invariant on a cusp of $\Gamma_1(p)$. By \cite[Proposition~3.8.3]{DS}), $\{a/c\}$ and $\{a'/c'\}$ (where $a,c$ are coprime integers and similarly for $a',c'$) are in the same cusp if and only if \[\begin{pmatrix}
         a'\\c'
     \end{pmatrix}\equiv\pm\begin{pmatrix}
         a+jc\\c
     \end{pmatrix}\mod p\] for some integer $j$.
  Hence, the lemma follows.
\end{proof}

\begin{theorem}\label{thm:lambda inv}
    Let $p$ be an odd prime number and $n \geq 1$ be an integer. If $\phi \in \BSymb_{\Gamma_1(p)}(\mathcal{V}_0(\ZZ / p \ZZ))$ is a $\mathcal{V}_0(\ZZ / p \ZZ)$-valued boundary symbol, then 
    \[\theta_{n,\phi}\equiv\mathcal{C}_{n,\phi}\sum_{\sigma\in\cG_n}\sigma\pmod p,\] where
    $$\mathcal{C}_{n,\phi}=\sum_{a \in (\ZZ / p \ZZ)^\times} ( \phi(\{\infty\})(1) - \phi(\{a/p\})(1)) . $$
    In particular, if $\mathcal{C}_{n,\phi}\not\equiv 0\pmod p$,
    then $\mu(\theta_{n,\phi})  =0$ and $\lambda(\theta_{n,\phi}) = p^n - 1$. \end{theorem}
\begin{proof}
    Let $C_{n,a} = \phi(\{\infty\})(1) - \phi(\{a/p^{n+1}\})(1)$ for $a \in (\ZZ / p^{n+1}\ZZ)^\times$. Consider the following group isomorphisms 
    \[
    (\ZZ / p^{n+1}\ZZ)^\times\cong (\Zp / p^{n+1}\Zp)^\times\cong \mu_{p-1}\times\frac{1+p\Zp}{1+p^{n+1}\Zp}\times\cong (\ZZ/p\ZZ)^\times\times\frac{1+p\ZZ}{1+p^{n+1}\ZZ}.
    \]
    We write $(\hat{a},\langle a\rangle) \in (\ZZ/p\ZZ)^\times\times (1+p\ZZ)/(1+p^{n+1}\ZZ)$ for the image of $a$.
    
    Recall that $\gamma_n$ is a fixed generator of the group $\cG_n=\Gal(K_n/\QQ)$, $T=\gamma_n-1$ is an element in the group ring $\Zp[\cG_n]$ and $\sigma_a\in G_{n+1}=\Gal(\QQ(\mu_{p^{n+1}})/\QQ)$ is the element given by $\zeta\mapsto \zeta^a$. Let $\tilde\sigma_a\in \cG_n$ be the natural image of $\sigma_a$ under the projection $G_{n+1}\twoheadrightarrow \cG_n$. We define $0 \leq i_{n,a} \leq p^n-1$ to be the unique integer such that $\tilde\sigma_a= \gamma_n^{i_{n,a}} $. Then $i_{n,a}=i_{n,b}$ if and only if $\langle a\rangle =\langle b\rangle$.
    We can thus deduce the following congruences modulo $p$:
    \begin{align*}
       \theta_{n,\phi} & \equiv \sum_{a \in (\ZZ / p^{n+1}\ZZ)^\times} C_{n,a} \cdot (1 + T)^{i_{n,a}}\\
       & \equiv \sum_{\hat a \in (\ZZ / p \ZZ)^\times} \left(\sum_{\langle a\rangle  \in \frac{1+p\ZZ}{1+p^{n+1}\ZZ}} C_{n,a} \cdot (1 + T)^{i_{n,a}}\right)\\
       & \equiv \sum_{\hat a \in (\ZZ / p \ZZ)^\times} C_{0,\hat a} \sum_{j=0}^{p^n - 1} (1 + T)^j\\
       &\equiv \mathcal{C}_{n,\phi}\sum_{\sigma\in\cG_n}\sigma,
   \end{align*}
    where the third congruence follows from Lemma~\ref{lem:cusps}.
   
In particular,
\[
       \theta_{n,\phi} \equiv \mathcal{C}_{n,\phi}\frac{(1 + T)^{p^n} - 1}{T} \equiv \mathcal{C}_{n,\phi}T^{p^n-1} .
\]
Therefore, if $\mathcal{C}_{n,\phi}\not\equiv 0\pmod p$, we have $\mu(\theta_{n,\phi}) = 0$ and $\lambda(\theta_{n,\phi}) = p^n - 1$, which concludes the proof of the theorem.
\end{proof}

Theorem~\ref{thm:lambda57-Eisen-intro} follows from Theorem~\ref{thm:lambda inv} once we verify that the boundary symbols $\phi_p$  (see the description of the values of $\phi_p$ given in the proof of Theorem~\ref{thm:congmodsymb}) satisfy 
$$\sum_{a \in (\ZZ / p \ZZ)^\times} ( \phi_p(\{\infty\})(1) - \phi_p(\{a/p\})(1)) \not\equiv 0 \quad (\text{mod } p)$$ 
Indeed, we can calculate explicitly that
$$\sum_{ a \in (\ZZ / 5 \ZZ)^\times } ( \phi_{0,\omega_5^2, \mathbbm{1}}(\{\infty\})(1) - \phi_{0,\omega_5^2,\mathbbm{1}}(\{ a/5 \})(1)) \equiv 3 \not\equiv 0 \pmod{5},$$
and
    $$\sum_{ a \in (\ZZ / 7 \ZZ)^\times } ( \phi_{0,\omega_7^4, \mathbbm{1}}(\{\infty\})(1) - \phi_{0,\omega_7^4,\mathbbm{1}}(\{ a/7 \})(1)) \equiv 5 \not\equiv 0 \pmod{7},$$
as required.
We are now ready to prove Theorem~\ref{thm:lambda-delta-57}.

\begin{corollary}\label{cor:p=5,7}
    For $p \in \{5,7\}$, we have $$\mu(\theta_{n,\Delta})=0,\quad\lambda(\theta_{n,\Delta}) = p^n - 1.$$
\end{corollary}

\begin{proof}
   Theorem~\ref{thm:congmodsymb} tells us that for $p \in \{5,7\}$, 
   $$\theta_{n,\Delta} \equiv c_p \theta_{n,\phi_p} \pmod{p},$$
   where $\phi_5 = \phi_{0, \omega_5^2, \mathbbm{1}}$ and $\phi_7 = \phi_{0, \omega_7^4, \mathbbm{1}}$ are boundary symbols and ${c_p \in (\ZZ / p \ZZ)^\times}$. Therefore, it follows that $$\mu(\theta_{n,\Delta}) = \mu(\theta_{n,\phi_p})=0,\quad \lambda(\theta_{n,\Delta}) = \lambda(\theta_{n,\phi_p}).$$
   Hence, the corollary follows from Theorem~\ref{thm:lambda57-Eisen-intro}.
\end{proof}

A similar calculation can be performed for the case $p=3$, proving Theorem~\ref{thm-intro:3}:

\begin{theorem}\label{thm:3}
    For every positive integer $n$, the Iwasawa invariants of the $3$-adic Mazur--Tate elements are given by 
    $$\mu(\theta_{n,\Delta}) = 3,\quad \lambda(\theta_{n,\Delta}) = 3^n - 2.$$
\end{theorem}

\begin{proof}
    Let $\phi_9$ be the boundary symbol as in Theorem~\ref{thm:congmodsymb3}. For $a \in (\ZZ/3^{n+1} \ZZ)^\times$, let $C'_{n,a} =  \phi_9(\{\infty\})(1) - \phi_9(\{a/3^{n+1} \})(1).$ For all $a \in (\ZZ/3^{n+1} \ZZ)^\times$, one can verify that $C'_{n,a} \in 3\ZZ / 9 \ZZ$. Thus, it makes sense to define $C_{n,a} = \frac{1}{3}C'_{n,a}$ and to consider $C_{n,a}$ as an element of $\ZZ / 3 \ZZ$.

    For each $a\in(\ZZ/3^{n+1}\ZZ)^\times$, define $0 \leq i_{n,a} \leq 3^n - 1$ to be the unique integer such that 
    $$\frac{a}{\omega_3(a)}\equiv 4^{i_{n,a}}\pmod { 3^{n+1}}.$$ 
    Note that this is well defined as $4$ is a generator of the cyclic group $(1+3\ZZ)/(1+3^{n+1}\ZZ)$, to which $a/\omega_3(a)$ belongs.
    
    Recall that $\pi_n:G_{n+1}\twoheadrightarrow \cG_n\cong\ZZ/3^n\ZZ$ is the natural projection map. Our discussion above tells us that the group $\cG_n$ is generated by $\pi_n(\sigma_4)$. Furthermore,    $$\pi_n(\sigma_a)=\pi_n(\sigma_{-a})=\pi_n(\sigma_4)^{i_{n,a}}.$$
    It follows from Theorem~\ref{thm:congmodsymb3} and the isomorphism given by \eqref{eq:iso-81} that 
    \begin{align*}
        \frac{1}{9}\theta_{n, \Delta} & \equiv \sum_{a \in (\ZZ / 3^{n+1}
        \ZZ)^\times} C'_{n,a} (1 + T)^{i_{n,a}}\\
         & \equiv \sum_{j = 0}^{3^n - 1} (C'_{n, 4^{j}} + C'_{n, -4^{j }}) (1 + T)^j\pmod 9.
        \end{align*}
        Thus, \[
        \frac{1}{27}\theta_{n, \Delta}  \equiv \sum_{j = 0}^{3^n - 1} (C_{n, 4^{j}} + C_{n, -4^{j }}) (1 + T)^j\pmod 3.
        \]

Similarly to the proof of Lemma~\ref{lem:cusps}, it follows from \cite[Proposition~3.8.3]{DS} and the definition of $\phi_9$ that if $a$ is an integer coprime to $3$, then 
$\{a/3^{n+1}\}$ is in the same cusp of $\Gamma_1(27)$ as $\{r/9\}$ where $1\le r\le 8$ is the unique integer such that $a\equiv r\mod 9$. Since $\phi_9$ is invariant on a cusp of $\Gamma_1(27)$,  the sum in the last congruence above can be rewritten as:\begin{align*}
     \ & (C_{n,1} + C_{n, 8})\left[ \sum_{\substack{j = 0 \\ j \equiv 0 \text{ (mod }3) }}^{3^n - 1} (1 + T)^j \right] + (C_{n,4} + C_{n, 5}) \left[ \sum_{\substack{j = 0 \\ j \equiv 1 \text{ (mod }3)}}^{3^n - 1} (1 + T)^j \right] \\
        & + (C_{n,2} + C_{n, 7}) \left[ \sum_{\substack{j = 0 \\ j \equiv 2 \text{ (mod } 3)}}^{3^n - 1} (1 + T)^j \right]\\
         \equiv\ & (0 + 0) \left[ \sum_{\substack{j = 0 \\ j \equiv 0 \text{ (mod } 3)}}^{3^n - 1} (1 + T)^j \right] + (2 + 2) \left[ \sum_{\substack{j = 0 \\ j \equiv 1 \text{ (mod } 3)}}^{3^n - 1} (1 + T)^j \right] \\
        & + (1 + 1) \left[ \sum_{\substack{j = 0 \\ j \equiv 2 \text{ (mod } 3)}}^{3^n - 1} (1 + T)^j \right]\\
         \equiv\ & (1+T)\cdot \frac{(1 + T)^{3^n} - 1}{T^3} +2 \cdot  (1 + T)^2 \cdot \frac{(1 + T)^{3^n} - 1}{T^3} \\
         \equiv\  & 2 \cdot  T^{3^n - 2}(1+T)\pmod3.
\end{align*}
    Therefore, we conclude that $\mu \left(\frac{1}{27}\theta_{n,\Delta}\right)=0$ and $ \lambda \left(\theta_{n,\Delta}\right) = 3^n - 2$  for all $n \geq 1$, as desired.
\end{proof}

\section{Appendix: Data used in the proof of Theorem~\ref{thm:congmodsymb3}}

\begin{longtable}[c]{| c | c | c |}
    \caption{Values of $\beta_9{\varphi^+_{\Delta}}$ and $\phi_9$ on $S_{27}$ \label{tab:long}}\\

    \hline
    \multicolumn{3}{| c |}{Beginning of Table\ref{tab:long}}\\
    \hline
    $\{ r\} - \{ s \} \in S _{27}$ & $\beta_9{\varphi^+_{\Delta}}(\{ r\} - \{ s \})$ & $\phi_9(\{r \})(1) - \phi_9(\{ s \})(1)\pmod{9}$ \\
    \endfirsthead
    
    \hline
    \multicolumn{3}{|c|}{Continuation of Table \ref{tab:long}}\\
    \hline
    $\{ r\} - \{ s \} \in S _{27}$ & $\alpha_9\overline{\varphi^+_{\Delta}}(\{ r\} - \{ s \})$ & $\phi_9(\{r \})(1) - \phi_9(\{ s \})(1)\pmod{9}$ \\
    \hline
    \hline
    \endhead

    \hline
    \endfoot

    \hline
    \multicolumn{3}{| c |}{End of Table\ref{tab:long}}\\
    \hline
    \endlastfoot

     \hline
     \hline
     $\{-2/27 \} - \{ -1/14\}$ & $5$ & $5$ \\
     \hline
     $\{5/27 \} - \{ 3/16\}$ & $2$ & $2$ \\
     \hline
     $\{-8/27 \} - \{ -5/17\}$ & $8$ & $8$ \\
     \hline
     $\{-10/27 \} - \{ -7/19\}$ & $8$ & $8$ \\
     \hline
     $\{4/27 \} - \{ 3/20\}$ & $2$ & $2$ \\
     \hline
     $\{11/27 \} - \{ 9/22\}$ & $5$ & $5$ \\
     \hline
     $\{7/27 \} - \{ 6/23\}$ & $5$ & $5$ \\
     \hline
     $\{-2/9 \} - \{ -5/23\}$ & $5$ & $5$ \\
     \hline
     $\{-3/17 \} - \{ -4/23\}$ & $3$ & $3$ \\
     \hline
     $\{7/18 \} - \{ 9/23\}$ & $5$ & $5$ \\
     \hline
     $\{-5/19 \} - \{ -6/23\}$ & $3$ & $3$ \\
     \hline
     $\{-12/41 \} - \{ -7/24\}$ & $3$ & $3$ \\
     \hline
     $\{-9/43 \} - \{ -5/24\}$ & $0$ & $0$ \\
     \hline
     $\{-5/17 \} - \{ -7/24\}$ & $6$ & $6$ \\
     \hline
     $\{-4/19 \} - \{ -5/24\}$ & $6$ & $6$ \\
     \hline
     $\{-2/47 \} - \{ -1/24\}$ & $0$ & $0$ \\
     \hline
     $\{-13/27 \} - \{ -12/25\}$ & $2$ & $2$ \\
     \hline
     $\{-9/32 \} - \{ -7/25\}$ & $3$ & $3$ \\
     \hline
     $\{-1/6 \} - \{ -4/25\}$ & $3$ & $3$ \\
     \hline
     $\{-2/7 \} - \{ -7/25\}$ & $0$ & $0$ \\
     \hline
     $\{-4/9 \} - \{ -11/25\}$ & $2$ & $2$ \\
     \hline
     $\{-3/37 \} - \{ -2/25\}$ & $6$ & $6$ \\
     \hline
     $\{-4/11 \} - \{ -9/25\}$ & $0$ & $0$ \\
     \hline
     $\{5/14 \} - \{ 9/25\}$ & $3$ & $3$ \\
     \hline
     $\{5/42 \} - \{ 3/25\}$ & $6$ & $6$ \\
     \hline
     $\{7/16 \} - \{ 11/25\}$ & $0$ & $0$ \\
     \hline
     $\{5/18 \} - \{ 7/25\}$ & $2$ & $2$ \\
     \hline
     $\{3/19 \} - \{ 4/25\}$ & $6$ & $6$ \\
     \hline
     $\{15/47 \} - \{ 8/25\}$ & $0$ & $0$ \\
     \hline
     $\{5/21 \} - \{ 6/25\}$ & $3$ & $3$ \\
     \hline
     $\{7/22 \} - \{ 8/25\}$ & $3$ & $3$ \\
     \hline
     $\{11/23 \} - \{ 12/25\}$ & $3$ & $3$ \\
     \hline
     $\{1/27 \} - \{ 1/26\}$ & $8$ & $8$ \\
     \hline
     $\{1/3 \} - \{ 9/26\}$ & $3$ & $3$ \\
     \hline
     $\{-41/355 \} - \{ -3/26\}$ & $6$ & $6$ \\
     \hline
     $\{-1/5 \} - \{ -5/26\}$ & $6$ & $6$ \\
     \hline
     $\{151/357 \} - \{ 11/26\}$ & $6$ & $6$ \\
     \hline
     $\{-3/7 \} - \{ -11/26\}$ & $3$ & $3$ \\
     \hline
     $\{69/359 \} - \{ 5/26\}$ & $0$ & $0$ \\
     \hline
     $\{1/9 \} - \{ 3/26\}$ & $8$ & $8$ \\
     \hline
     $\{-125/361 \} - \{ -9/26\}$ & $0$ & $0$ \\
     \hline
     $\{-3/11 \} - \{ -7/26\}$ & $3$ & $3$ \\
     \hline
     $\{-14/363 \} - \{ -1/26\}$ & $0$ & $0$ \\
     \hline
     $\{14/365 \} - \{ 1/26\}$ & $6$ & $6$ \\
     \hline
     $\{4/15 \} - \{ 7/26\}$ & $0$ & $0$ \\
     \hline
     $\{127/367 \} - \{ 9/26\}$ & $3$ & $3$ \\
     \hline
     $\{-2/17 \} - \{ -3/26\}$ & $0$ & $0$ \\
     \hline
     $\{-71/369 \} - \{ -5/26\}$ & $8$ & $8$ \\
     \hline
     $\{8/19 \} - \{ 11/26\}$ & $0$ & $0$ \\
     \hline
     $\{-157/371 \} - \{ -11/26\}$ & $3$ & $3$ \\
     \hline
     $\{4/21 \} - \{ 5/26\}$ & $6$ & $6$ \\
     \hline
     $\{43/373 \} - \{ 3/26\}$ & $6$ & $6$ \\
     \hline
     $\{-8/23 \} - \{ -9/26\}$ & $6$ & $6$ \\
     \hline
     $\{-101/375 \} - \{ -7/26\}$ & $3$ & $3$ \\
     \hline
     $\{-1/25 \} - \{ -1/26\}$ & $3$ & $3$ \\
\end{longtable}

\subsection*{Data availability statement}
The authors declare that the data supporting the findings of this study are available within the paper, its source code is available on \href{https://github.com/anthonydoyon/Ramanujan-s-tau-and-MT-elts}{https://github.com/anthonydoyon/Ramanujan-s-tau-and-MT-elts}.
\subsection*{Conflict of interest statement}All authors have no conflicts of interest.
\bibliographystyle{amsalpha}
\bibliography{references}

\end{document}